\documentclass[11pt,a4paper,reqno]{amsart}
\usepackage{geometry}
\usepackage{amsmath,amsthm,amssymb}
\usepackage{microtype}

\usepackage{hyperref}
\usepackage{cleveref}
\newtheorem{theorem}{Theorem}[section]
\theoremstyle{plain}
\newtheorem{corollary}[theorem]{Corollary}
\newtheorem{definition}{Definition}[section]
\newtheorem{example}[theorem]{Example}
\newtheorem{lemma}[theorem]{Lemma}
\newtheorem{proposition}[theorem]{Proposition}
\numberwithin{equation}{section}
\newtheorem{remark}[theorem]{Remark}

\renewcommand{\Re}{\operatorname{Re}}
\renewcommand{\Im}{\operatorname{Im}}

\usepackage{enumerate}

\newcommand{\R}{\mathbb{R}}

\newcommand{\C}{\mathbb{C}}

\newcommand{\aba}{\bar{\alpha}}
\newcommand{\bba}{\bar{\beta}}

\newcommand{\wba}{\bar{w}}

\newcommand{\zba}{\bar{z}}

\DeclareMathOperator{\rk}{rk}
\DeclareMathOperator{\Aut}{Aut}
\newcommand{\heis}{\mathbb{H}^3}
\newcommand{\heisN}{\mathbb{H}^{2N+1}}
\newcommand{\heisn}{\mathbb{H}^{2n+1}}

\begin{document}
\title{On CR maps between hyperquadrics and Winkelmann hypersurfaces}
\author{Michael Reiter}
\address{Fakultät für Mathematik, Universität Wien, Oskar-Morgenstern-Platz 1, 1090 Wien, Austria}
\email{m.reiter@univie.ac.at}
\author{Duong Ngoc Son}
\address{Faculty of Fundamental Sciences, PHENIKAA University, Hanoi 12116, Vietnam}
\email{son.duongngoc@phenikaa-uni.edu.vn}
\thanks{Part of this work was done while the second-named author was supported by the Vietnam Institute for Advanced Study in Mathematics, Hanoi, Vietnam (July--September, 2023).}
\date{17 January 2024.}
\begin{abstract}
In this paper, we study CR maps between hyperquadrics and Winkelmann hypersurfaces. Based on a previous study on the CR Ahlfors derivative of Lamel--Son and a recent result of Huang--Lu--Tang--Xiao on CR maps between hyperquadrics, we prove that a transversal CR map from a hyperquadric into a hyperquadric or a Winkelmann hypersurface extends to a local holomorphic isometric embedding with respect to certain K\"ahler metrics if and only if the Hermitian part of its CR Ahlfors derivative vanishes on an open set of the source. Our proof is based on relating the geometric rank of a CR map into a hyperquadric and its CR Ahlfors derivative.
\end{abstract}
\maketitle

\section{Introduction}

In this paper, we study CR maps between real hypersurfaces in complex spaces of different dimensions. We are particularly interested in those with symmetry algebras of high dimension. It is known from the work of Cartan \cite{cartan1935domaines}, Tanaka \cite{tanaka1962}, and Chern--Moser \cite{chern1974} that for a Levi-nondegenerate hypersurface in \(\C^{n+1}\), its symmetry algebra has dimension at most \(n^2 + 4n +3\), see also \cite{kruglikov}. This maximal value is attained for spheres and hyperquadrics: For any $\ell$ and $n$, \(0 \leq \ell \leq n/2\), the \textit{hyperquadric of signature \(\ell\)} is defined by \(\rho_{\mathbb{H}_{\ell}^{2n+1}}=0\), where
\[ 
\rho_{\mathbb{H}_{\ell}^{2n+1}}:=\Im (w) - \sum_{k=1}^{n} \epsilon_k |z_k|^2, \qquad (z,w) \in \mathbb{C}^{n} \times \mathbb{C}.
\]
Here \(\epsilon_k = -1\) for \(1\leq k \leq \ell\) (there is no such \(k\) if \(\ell = 0\)) and \(\epsilon_k = 1\) for \(k=\ell,\dots n\).
When \(\ell = 0\), the hyperquadric is the Heisenberg hypersurface which is strictly pseudoconvex and CR diffeomorphic to the sphere
\[ 
\mathbb{S}^{2n+1} : = \left\{(z_1,\dots, z_{n+1}) \in \mathbb{C}^{n+1} \mid |z_1|^2 + |z_2|^2 + \cdots + |z_{n+1}|^2 -1 = 0\right\}
\]
with one point deleted.

The sphere and hyperquadrics are important models of Levi-nondegenerate hypersurfaces. The study of CR maps between spheres (sphere maps for short) has been addressed in a lot of articles, see, e.g., \cite{Alexander77, AHJY16, Webster79b, faran1982maps, dangelo1988proper, CJX06, Lebl11, AHJY16, CJY19}, and the book by D'Angelo \cite{DAngeloBook21} for more references.
In the case of hyperquadrics of positive signatures, the CR maps, which we refer to as \emph{hyperquadric maps}, have also been studied extensively. We mention, for example, an important paper by Baouendi--Huang \cite{baouendi2005super} which establishes the super-rigidity for such CR maps. After the publication of this paper, there 
is a large literature devoted to the study of CR maps between hyperquadrics, see \cite{BEH08, BEH11, DAngelo07a, LP11, GLV14, DAngelo15, Reiter16a, GN21a}.

An useful invariant for the study of hyperquadric maps (including sphere maps) is the \textit{geometric rank} for such maps, which was first defined by Huang for the sphere case in \cite{Huang99} and later for hyperquadrics of general signatures in \cite{huang2020boundary}. This invariant has played an important role in many papers; we mention here, for examples, the papers \cite{Huang99, huang2001mapping, Huang03, Hamada05, HJX06, Ji10, HJY14}.  We observe that the definition of the geometric rank relies on the possibility of normalizing the map in question using the (large and explicit) CR automorphism group of the target. It is not clear at the moment how can we generalize this definition to the case of CR maps into a target that is \textit{not} a hyperquadric.

Another spherical invariant for sphere maps was introduced by Lamel--Son \cite{lamel2021cr}, namely the \emph{CR Ahlfors derivative} for a CR map \(H\), denoted by \(\mathcal{A}(H)\). We refer the readers to \cite{lamel2021cr} for the origin and motivation of this notion and to \cite{stowe} for its conformal counterpart. Briefly, the CR Ahlfors derivative of a transversal CR immersion is a 2-tensor, which can be defined for the case of arbitrary hypersurfaces. If \(H\) is a sphere map, as first observed by Lamel--Son, the rank of the Hermitian part \(\mathcal{A}_{\alpha\bba}(H)\) of its CR Ahlfors derivative agrees with its geometric rank. It is not unexpected that we observe the same invariant property as well as its agreement with the geometric rank of the CR Ahlfors derivative in the case of (transversal) hyperquadric maps. But more interestingly, we discover an example of hypersurfaces such that this invariant property also holds. Specifically, we find that the CR Ahlfors derivative for CR maps between a hyperquadric and a \emph{Winkelmann hypersurface} (see below) also possesses an invariant property with respect to composition with CR automorphisms.

We should point out that a Winkelmann hypersurface, like a hyperquadric, also has symmetry algebra of quite large dimension. This fact is one of the reasons we consider the Winkelmann hypersurface as a target. In fact, it has the \textit{second} largest possible dimension of the symmetry algebra (among those hypersurfaces of the same dimension), referred to as the sub-maximal dimension case. More specifically, for hypersurfaces in \(\mathbb{C}^{n^{\prime}+2}\), under the Levi-nondegeneracy condition, the sub-maximal dimension is \((n'+1)^2+4\), which is attained for the hypersurface $\mathcal{W}^{2n'+3}_{\ell'} \subset \mathbb{C}^{n^{\prime}+2}$ (\(1\leq \ell' \leq n'\)) given by \cite{kruglikov}
\begin{equation}\label{ef}
\rho_{\mathcal{W}^{2n^{\prime }+3}_{\ell'}}:=\Im(w + \bar{z}_{n'} \zeta) - |z_{n'}|^4 - \sum_{k=1}^{n'-1} \epsilon_k'|z_k|^2 = 0, \quad n^{\prime} \geq 1,
\end{equation}
where \((z_1,\dots, z_{n'}, \zeta, w)\) are holomorphic coordinates in \(\C^{n^{\prime}+2}\) and \(\epsilon_k' = -1\) for \( 1\leq k \leq \ell' - 1\) (there is no such \(k\) if \(\ell' = 1\)) and \(\epsilon_k' = 1\) for \(k=\ell',\dots n'-1\), has exactly $(n'+1)^2 + 4$ independent symmetries. This is a generalization of a hypersurface in \(\mathbb{C}^3\) (corresponding to \(n^{\prime}=1\)) of Levi signature \(\ell'=1\), which first appears in \cite[page 146]{winkelmann}. 
It plays an important role in the study of homogeneous CR manifolds in \(\mathbb{C}^3\), see e.g.  \cite{doubrov2021} and the references therein. We shall refer to the hypersurface $\mathcal{W}^{2n'+3}_{\ell'} \subset \mathbb{C}^{n'+2}$ for arbitrary \(n'\geq 1\) and \(1 \leq \ell' \leq n'\) as the Winkelmann hypersurface of dimension \(2n'+3\) and of Levi signature~\(\ell'\).

The classification problem for CR maps from and into a Winkelmann hypersurface is surely interesting, but not yet studied. In this paper, from the invariant property of the CR Ahlfors derivative, we can introduce a notion of geometric rank for CR maps from and into a Winkelmann hypersurface and characterize being the restriction of a local holomorphic isometric embedding of certain K\"ahler metrics in terms of the vanishing of the CR Ahlfors derivative. This characterization is analogous to a recent result of Huang--Lu--Tang--Xiao \cite{huang2020boundary}. In fact, we shall use their result in our proof. More specifically, the main result of this paper is as follows.
\begin{theorem}\label{thm:main1} Let \(U \subset \mathbb{H}^{2n+1}_{\ell}\) be an open subset of the hyperquadric and \(H\colon U \to\mathcal{W}_{\ell'}^{2n'+3}\) a CR transversal smooth CR map. Let \(\mathcal{A}_{\alpha\bba}(H)\) be the Hermitian part of the CR Ahlfors tensor of \(H\) with respect to ``standard'' pseudo-Hermitian structures of the source and target. Then \(\mathcal{A}_{\alpha\bba}(H) = 0\) on \(U\) if and only if \(H\) extends to a local isometric embedding of the indefinite ``canonical'' K\"ahler metrics.
\end{theorem}
In Theorem~\ref{thm:main1}, the CR Ahlfors tensor should be defined with respect to special pseudo-Hermitian structures on the source and target.
Specifically, on the hyperquadric we consider the usual pseudo-Hermitian structure
\[
\theta = i \bar{\partial} \rho_{\mathbb{H}_{\ell}^{2n+1}},
\]
which has vanishing pseudo-Hermitian torsion and curvature. Similarly, on the Winkelmann hypersurface \(\mathcal{W}_{\ell'}^{2n'+3}\), we consider
\[
\Theta = i \bar{\partial} \rho_{\mathcal{W}_{\ell'}^{2n'+3}}
\]
as the ``canonical'' pseudohermitian structure.

The indefinite K\"ahler metrics on the complements of the source and target are some kind
of canonical metrics. Precisely, we consider on the ``upper'' domain
\[
\mathcal{U}^{n+1}_{\ell} = \left\{p \in \C^{n+1} \mid \rho_{\mathbb{H}_{\ell}^{2n+1}} > 0\right\},
\]
equipped with the K\"ahler metric
\[
\omega_{\ell}^{n+1} = i\partial \bar{\partial} \log \left(\frac{1}{ \rho_{\mathbb{H}_{\ell}^{2n+1}}}\right),
\]
which is a K\"ahler metric of constant holomorphic sectional curvature. It is positive definite when \(\ell = 0\) and indefinite of signature \(\ell\) otherwise \((\ell > 0)\). 

Similarly, we consider the half space defined by the Winkelmann hypersurface
\[
\Omega^{n'+2}_{\ell'} =\left\{ (z_1, \dots, z_{n'}, \zeta, w) \in \C^{n'+2} \mid \rho_{\mathcal{W}_{\ell'}^{2n'+3}} > 0\right\}.
\]
On \(\Omega\) we consider the K\"ahler metric given by 
\begin{equation}\label{e:winkelman-kahler}
\omega_{\mathcal{W}_{\ell'}^{2n'+3}} : = i \partial\overline{\partial} \log \left(\frac{1}{\rho_{\mathcal{W}_{\ell'}^{2n'+3}}}\right).
\end{equation}
It is interesting to note that \(\omega_{\mathcal{W}_{\ell'}^{2n'+3}}\) is an indefinite K\"ahler-Einstein metric whose Ricci curvature equals $n'+2$. Our main theorem says that if \(H\) is CR transversal, preserves sides, and has CR Ahlfors derivative vanishing on an open set of the hyperquadric, $H$ must extend to a local isometric embedding from  \(\mathcal{U}^{n+1}_{\ell}\) into \(\Omega^{n'+2}_{\ell'}\), equipped with these metrics. The side preserving condition is natural and often appear in the study of CR maps of Levi-nondegenerate real hypersurfaces with positive signature as in, e.g., \cite{baouendi2005super}. If the map is side reversing, then we should change the side to which the map extends.

The proof of the main theorem relies on a version of a result of Huang--Lu--Tang--Xiao \cite{huang2020boundary}, stated in Corollary \ref{cor33}, and the equality of the rank of the Hermitian part of the CR Ahlfors and the geometric rank (Corollary \ref{cor:equality}).

The paper is organized as follows. In Section 2, we briefly review some basic facts related to CR Schwarzian and Ahlfors derivatives. In Section 3, we revisit, for the sake of completeness, the case of CR maps between spheres and hyperquadrics. We shall give a detailed proof of the coincidence of the geometric rank of a CR map and the rank of the Hermitian part of its CR Ahlfors tensor. Applying a recent work of Huang--Lu--Tang--Xiao \cite{huang2020boundary}, we give a characterization of CR maps with vanishing Ahlfors derivative. In Section 4, we will characterize all CR M\"obius pseudo-Hermitian structures with respect to a special choice of a pseudo-Hermitian structure. The notion of CR M\"obius structures was studied in \cite{son2018}; they are related to the invariant property of the CR Ahlfors derivative. We will also show that the CR Ahlfors derivative for a CR map into \(\mathcal{W}_{\ell'}^{2n'+3}\) possesses an invariant property, with respect to composing with an automorphism of \(\mathcal{W}_{\ell'}^{2n'+3}\). In Section 5, we characterize the CR maps from a sphere or hyperquadric into a Winkelmann hypersurface with vanishing CR Ahlfors derivative. Finally, in Section 6, we will provide various examples of CR maps from a hyperquadric into a Winkelmann hypersurface.

\section{The CR Schwarzian and Ahlfors derivative}
In this section, we briefly review the notions of CR Schwarzian and Ahlfors derivatives and their basic properties.

Let \((M,\theta)\) and \((N,\eta)\) be pseudo-Hermitian manifolds and let \(F \colon M\to N\) be a CR transversal CR immersion. Then 
\[ F^{\ast} \eta = \varphi\, \theta\]
for some function \(\varphi\) on \(M\), which is non-vanishing because of the transversality of \(F\). For simplicity, we suppose that \(\varphi > 0\) and set \(u = \log \varphi\). Following Lamel--Son \cite{lamel2021cr} we define
\begin{equation}
\mathcal{H}_{\theta} (v) = \operatorname{Sym} \nabla\nabla v - \partial_b v \otimes \partial_b v - \bar{\partial}_b v \otimes \bar{\partial}_b v + \frac{1}{2}\left| \bar{\partial}_b v \right|^2 L_{\theta},
\end{equation}
where \(\nabla\) denotes the Tanaka--Webster connection of \((M,\theta)\).
If \(F\) is a local CR diffeomorphism, then the CR Schwarzian of \(F\) is given by \cite{son2018}
\begin{equation}\label{e:crschwarzian}
\mathcal{S}_{\theta}(F) = B_{\theta} (\log \varphi),
\end{equation}
where
\begin{equation}\label{e:crschwarzian2}
B(v) = B_{\theta}(v) = 2 \mathcal{H}_{\theta}(v) - \frac{1}{n}\left(\Delta_b v - 2n |\bar{\partial}_b v|^2\right) L_{\theta}.
\end{equation}
In fact, the formula for \(B\) is simpler. Specifically, we have in a local frame of \(M\),
\[
B(v)_{\alpha\bar{\beta}} = v_{\alpha,\bar{\beta}} + v_{\bar{\beta},\alpha} - \frac{1}{n} \left(v_{\gamma,}{}^{\gamma}+ v_{\bar{\gamma},}{}^{\bar{\gamma}}\right) h_{\alpha\bar{\beta}},
\]
and
\[
B(v)_{\alpha\beta} = 2v_{\alpha,\beta} - 4 v_{\alpha} v_{\beta}, \quad B(v)_{\bar{\alpha}\bar{\beta}} = \overline{B(v)_{\alpha\beta}}.
\]
Here, an indices preceded by a comma indicates the covariant derivative with respect to the Tanaka--Webster connection. Lower and raising indices are done by the Levi matrices; see \cite{son2018} and \cite{lamel2021cr} for more details.

We should point out that the formula for \(B(v)_{\alpha\bba}\) appears a long time ago in related and seemingly different situations \cite{lee1988pseudo}. It is well-known that \(B(v)_{\alpha\bba} = 0\) for a real-valued function \(v\) if and only if \(v\) is a CR pluriharmonic function (that is, \(v\) is locally the real-part of a CR function).

The CR Ahlfors derivative is a generalization of the CR Schwarzian derivative for CR immersions, that is, for the CR maps into higher dimensional target. Its definition is a bit more complicated: It involves terms related to the CR second fundamental form of the immersion \(F\). Since we only need a formula for the CR Ahlfors derivative in the special case when the hypersurfaces are defined by an approximate Feffermann defining function (which is the case of hyperquadrics and Winkelmann hypersurface), we shall not reproduce the geometric construction of the CR Ahlfors derivative here. We refer the reader to \cite{lamel2021cr} for details. 

An important property of the CR Schwarzian and Ahlfors derivative is the ``chain rule'', which is described as follows:
\[
\mathcal{A}(G \circ F) = \mathcal{A}(F) + F^{\ast} \mathcal{A}(G),
\]
holds for a chain of CR immersions \(F\) and \(G\) and its compositions. This is \cite[Theorem 1.2]{lamel2021cr}. An interesting situation is when \(F\) or \(G\) has vanishing CR Ahlfors derivative, e.g., when \(F\) and \(G\) are CR automorphisms of a sphere or a hyperquadric. Specifically, if \(\gamma\) is a CR automorphism of the target sphere or hyperquadric (equipped with the usual pseudo-Hermitian structure), then
\[
\mathcal{A}(\gamma\circ F) = \mathcal{A}(F).
\]
Thus, the CR Ahlfors derivative provides an invariant for spherically equivalent sphere or hyperquadric CR maps. Moreover, if \(\phi\) is a CR automorphism of the source, then
\[
\mathcal{A}(F\circ \phi) = \phi^{\ast} \mathcal{A}(F) = e^{u\circ F} \mathcal{A}(F)
\]
for some smooth function \(u\). Hence, since \(e^{u\circ F}\) is non-vanishing, the rank of \(\mathcal{A}(F)_{\alpha\bba}\) is invariant with respect to the pre-composition with a CR automorphism of the source. We therefore conclude that the rank of the Hermitian part of the CR Ahlfors derivative is invariant with respect to compositions with CR automorphisms both of the source and target. One of the key facts in our paper is that this invariant property of the CR Ahlfors tensor also holds for the case of Winkelmann hypersurfaces of arbitrary dimensions and signatures; see Section~4.

\section{CR maps between hyperquadrics}
In this section, we collect several results on the CR Ahlfors derivative of CR maps between hyperquadrics. More specifically, we present a relation between the CR Ahlfors derivative and the geometric rank of the map. This relation is essentially the same as for the case of spheres, which was observed earlier by Lamel--Son. Based on this relation, we can restate a recent theorem of Huang--Lu--Tang--Xiao \cite{huang2020boundary} for the CR transversal maps between hyperquadrics with vanishing geometric rank. 

For the sake of completeness, we shall present further details. By the same construction as in \cite{lamel2021cr}, we can define the CR Ahlfors derivative for an arbitrary CR transversal map between Levi-nondegenerate CR manifolds. 
\begin{lemma}
Let \(\mathbb{H}_{\ell}^{2n+1}\) be a hyperquadric of signature \(\ell \geq 0\) in \(\mathbb{C}^{n+1}\) and let \(\gamma\) be a CR automorphism of \(\mathbb{H}_{\ell}^{2n+1}\). Then its CR Schwarzian derivative with respect to the standard pseudo-Hermitian structure \({\Theta}\) vanishes identically: \[ \mathcal{S}_{\Theta}(\gamma) = 0. \]
\end{lemma}
\begin{proof}
Observe that a CR automorphism of \(\mathbb{H}_{\ell}^{2n+1}\) is CR transversal and hence its Schwarzian derivative is well-defined. The vanishing of the CR Schwarzian follows directly from the explicit formula for the automorphisms (as provided in \cite{chern1974}) and direct calculations. We leave the details to the readers.
\end{proof}
Suppose that $F \colon U \to \mathbb{C}^{N+1}$ is a holomorphic map, with $F(\mathbb{H}_{\ell}^{2n+1}) \subset \mathbb{H}_{\ell^{\prime}}^{2N+1}$ transversally. Then there is a real-valued function $Q$ defined on a neighborhood of $0$, such that the following mapping equation holds
\begin{equation}\label{e:mapping}
\frac{G - \bar{G}}{2i} - \sum_{k=1}^{N} \epsilon_k|F^{k}|^2 = Q(z,w,\zba,\wba) \left(\frac{w-\bar{w}}{2i} - |z|^2_{\ell}\right),
\end{equation}
where \(Q\) does not vanish along \(\mathbb{H}_{\ell}^{2n+1}\). As proved by Lamel--Son \cite{lamel2021cr} in the strictly pseudoconvex case it holds that 
\begin{equation}\label{e:ahlfors}
\mathcal{A}(F)(Z_{\alpha}, Z_{\bba}) = \left\langle i \bar{\partial} \partial \log (Q) , Z_{\alpha} \wedge Z_{\bba} \right \rangle,
\end{equation}
provided that \(Q>0\) (that is, \(F\) preserves sides).

If \(Q<0\), we say that \(F\) reverses sides and, in the formula above, we can replace \(Q\) by \(-Q\); details are left to the readers.

We point out that \eqref{e:ahlfors} is a special case of a general formula which is valid for arbitrary defining functions of the hypersurfaces. When we choose a Feffermann defining function, which is the case of the spheres and hyperquadrics with ``usual'' defining function, the formula for the CR Ahlfors tensor greatly simplifies. On the other hand, \eqref{e:ahlfors} is very useful for computing the CR Ahlfors tensor, especially when \(Q\) can be computed explicitly. In many cases, finding the quotient \(Q\), which amounts to factorizing the left-hand side of \eqref{e:mapping} in the polynomial ring or the ring of power series, is very simple. However, in some situations, we may wish to avoid a division in the power series ring. In these situations, we can use the following result.
\begin{proposition} If \(H=(F^1, \dots, F^{N},G)\) is a holomorphic map that sends a hyperquadric in \(\mathbb{H}^{2n+1}_{\ell} \subset \mathbb{C}^{n+1}\) into \(\mathbb{H}^{2N+1}_{\ell^{\prime}} \subset \mathbb{C}^{N+1}\) transversally. There exists a non-vanishing real-valued function \(Q\) in a neighborhood of \(\mathbb{H}^{2n+1}_{\ell}\) such that
\begin{equation}
\rho_{\mathbb{H}^{2N+1}_{\ell^{\prime}}} \circ H = Q\, \rho_{\mathbb{H}^{2n+1}_{\ell}}
\end{equation}
Suppose that \(Q>0\) and \(\mathcal{A}(H)\) is the CR Ahlfors derivative of \(H\) with respect to the standard pseudo-Hermitian structures of the sources and target. Then, in the holomorphic frame \(\{Z_{\alpha}\}\), we have
\begin{align}\label{e:alhfors1}
\mathcal{A}_{\alpha\bba} (H) & = -\frac{2i}{Q}\sum_{k=1}^{N} \epsilon_k (Z_{\alpha}F^k _w) (Z_{\bba}\overline{F}^k) + \frac{4}{Q^2} \sum_{j,k=1}^{N} \epsilon_j\epsilon_kF^k_w (Z_{\bba} \overline{F}^k) \overline{F}^j_{\wba} (Z_{\alpha} F^j) \notag \\
& \qquad + \left( iG_{ww}+2\sum_{k=1}^{N} \epsilon_k \left( \overline{F}^k F^k_{ww}+ 2 |F^k_w|^2\right) \right) h_{\alpha\bba}.
\end{align}
\end{proposition}
\begin{proof}
We consider the vector field 
\[
\xi = 2i \frac{\partial}{\partial w}.
\]
Applying \(\xi\) to both sides of the mapping equation yields

\begin{equation}\label{e:a1}
G_w - 2i\sum_{k=1}^{N} \epsilon_k \overline{F}^k F^k _w = \xi(Q) \left(\frac{w-\bar{w}}{2i} - |z|^2_{\ell}\right) + Q.
\end{equation}

Applying \(\overline{\xi} \) to both sides of the just obtained equation \eqref{e:a1} and restricting to \(\mathbb{H}^{2n+1}_{\ell}\), we obtain
\[
4 \sum_{k=1}^{N} \epsilon_k |F^k _w|^2 \biggl|_{\mathbb{H}_{\ell}^{2n+1}} 
=
\xi(Q) + \overline{\xi} (Q)\biggl|_{\mathbb{H}_{\ell}^{2n+1}} .
\]
From these two formulas, we can apply the result of Lamel--Son \eqref{e:ahlfors} to obtain the desired formula. Specifically, from \eqref{e:a1}, we have
\begin{equation}\label{e:Q}
Q\bigl|_{\mathbb{H}_{\ell}^{2n+1}} = G_w - 2i\sum_{k=1}^{N} \epsilon_k \overline{F}^k F^k _w\biggl|_{\mathbb{H}_{\ell}^{2n+1}}.
\end{equation}
Applying \(Z_{\bba}\), which is tangent to \({\mathbb{H}_{\ell}^{2n+1}}\), to both sides of \eqref{e:Q}, we have the following holds on \(\mathbb{H}^{2n+1}_{\ell}\):

\begin{equation}\label{e:a2}
Z_{\bba} Q = - 2i\sum_{k=1}^{N} \epsilon_k F^k _w Z_{\bba}\overline{F}^k .
\end{equation}
Applying \(Z_{\alpha}\) to \eqref{e:a2}, we have
\[
Z_{\alpha}Z_{\bba} Q = - 2i\sum_{k=1}^{N} \epsilon_k (Z_{\alpha}F^k _w) (Z_{\bba}\overline{F}^k) - 2i \sum_{k=1}^{N} \epsilon_k F^k _w Z_{\alpha}Z_{\bba}\overline{F}^k. 
\]
But
\[
Z_{\alpha}Z_{\bba}\overline{F}^k = 2i \epsilon_{\beta} \delta_{\alpha}^{\beta} \overline{F}^k_{\wba} = 2ih_{\alpha\bba} \overline{F}^k_{\wba},
\]
and therefore,
\[
Z_{\alpha}Z_{\bba} Q = - 2i\sum_{k=1}^{N} \epsilon_k (Z_{\alpha}F^k _w) (Z_{\bba}\overline{F}^k) + 4\sum_{k=1}^{N} \epsilon_k |F^k _w|^2 h_{\alpha\bba}.
\]
On the other hand, applying \(Z_{\alpha}\) to both sides of \eqref{e:Q}, we have
\begin{equation}
Z_{\alpha} Q = Z_{\alpha} G_w - 2i \sum_{k=1}^{N} \epsilon_k \overline{F}^{k}(Z_{\alpha} F^k_{w}). 
\end{equation}
Applying \(Z_{\bba}\), we obtain
\begin{equation}
Z_{\bba}Z_{\alpha} Q 
=
Z_{\bba} Z_{\alpha} G_{w} - 2i \sum_{k=1}^{N} \epsilon_k (Z_{\bba} \overline{F}^k)(Z_{\alpha} F^k_w) - 2i \sum_{k=1}^{N} \epsilon_k \overline{F}^k (Z_{\bba}Z_{\alpha} F^k_w).
\end{equation}
Simplifying the right-hand side, we have
\begin{equation}
Z_{\bba}Z_{\alpha} Q 
=
2i h_{\alpha\bba} \left(G_{ww}- 2i \sum_{k=1}^{N} \epsilon_k \overline{F}^k F^k_{ww} \right) - 2i \sum_{k=1}^{N} \epsilon_k (Z_{\bba} \overline{F}^k)(Z_{\alpha} F^k_w) .
\end{equation}
Put \(u = \log Q\). Differentiating along \(Z_{\alpha}\) and applying the formula for the CR Ahlfors tensor \cite[Eq. (6.14)]{lamel2021cr} (notice that the terms related to the Fefferman determinant vanish), we have,
\begin{align}
\mathcal{A}_{\alpha\bba} (F) & = \frac{1}{2} \left(u_{\alpha,\bba} + u_{\bba,\alpha} + (\xi + \overline{\xi} )u\right) \notag \\
& = 
\frac{1}{2Q} \left(Z_{\alpha}Z_{\bba} Q + Z_{\bba} Z_{\alpha} Q\right) - \frac{1}{Q^2}\left(Z_{\alpha} Q\right) \left(Z_{\bba} Q\right) + \frac{1}{2Q} \left(\xi Q + \bar{\xi Q}\right) \notag \\
& = -\frac{2i}{Q}\sum_{k=1}^{N} \epsilon_k (Z_{\alpha}F^k _w) (Z_{\bba}\overline{F}^k) + \frac{4}{Q^2} \sum_{j,k=1}^{N} \epsilon_j\epsilon_kF^k_w (Z_{\bba} \overline{F}^k) \overline{F}^j_{\wba} (Z_{\alpha} F^j) \notag \\
& \qquad + \left( iG_{ww}+2\sum_{k=1}^{N} \epsilon_k \left( \overline{F}^k F^k_{ww}+ 2 |F^k_w|^2\right) \right) h_{\alpha\bba},
\end{align}
holding along \({\mathbb{H}_{\ell}^{2n+1}}\). We complete the proof.
\end{proof}

The case \(Q<0\) is similar and left to the readers.

Formula \eqref{e:alhfors1} is useful when the map satisfies certain conditions. Specifically, if \(H\) maps the origin to the origin and satisfies several normalization conditions, then the CR Ahlfors tensor of \(F\) at the origin can be computed from the coefficients of the Taylor series expansions of \(F\) up to second orders. For instance, if we assume that

\begin{equation}\label{e:normal1}
\frac{\partial G}{\partial w}\biggl|_0 =1, \quad \frac{\partial F^k}{\partial w}\biggl|_0 = \frac{\partial^2 G}{\partial w^2}\biggl|_0 = 0,\quad k =1,2,\dots, N,
\end{equation}
then \eqref{e:alhfors1}, evaluated at the origin, gives
\begin{equation}\label{e:ahlfors2}
\mathcal{A}(H)_{\alpha\bba}\biggl|_0 = -2i \sum_{k=1}^{N} \epsilon_k \frac{\partial ^2 F^k}{\partial w\partial z_{\alpha}} \overline{ \frac{\partial F^k}{\partial {z}_{\beta}}}\biggl|_0.
\end{equation}
Hence, we obtain the following result, which was observed earlier by Lamel--Son for the case of spheres.
\begin{corollary} \label{cor:equality}
Let $H \colon U \subset \mathbb{H}^{2n+1}_{\ell} \to \mathbb{H}^{2N+1}_{\ell'}$ be a transversal CR map and \(p \in U\). Then
\[ \rk\left(\mathcal{A}_{\alpha\bar{\beta}}(H)|p\right) = \rk_H(p). \]	
\end{corollary}
\begin{proof} As proved in \cite{huang2020boundary} for the case of \(\ell> 0\) and in \cite{Huang99} for the case of \(\ell = 0\), for each point \(p\), there are CR automorphisms \(\psi\) and \(\gamma\) such that \(\gamma(0) = p\) and \(\tilde{H}:=\psi^{-1} \circ H \circ \gamma\) has the following form: \(\tilde{H} = (F^{k})\), where
\begin{align}\label{e:normal}
F^k = 
\begin{cases}
z_\alpha + \frac{i}{2} a_{\alpha}(z) w + O_{wt}(4) \quad & k = \alpha \in \{1,2,\dots, n\} \\
\phi_l^{(2)}(z) + O_{wt}(3), & k = l \in \{n+1, \dots, N\}\\
w + O_{wt}(5) \quad  & k = N+1
\end{cases}
\end{align}
with 
\[
\langle \bar{z}, a(z) \rangle_{\ell} \|z\|^2_{\ell} = \|\phi^{(2)}(z)\|^2_{\tau}.
\]
Here \(\tau = \ell' - \ell \geq 0\) is the signature difference. (In the case \(\tau < 0\), we should consider side reversing maps, which is similar.) Moreover, \(a(z) = z A(p)\) for some Hermitian matrix \(A\). The rank of \(A\) is then an invariant of the spherically equivalent class of the map, which is called the \textit{geometric rank} of \(H\) at \(p\). In this form, \(\tilde{H}\) satisfies the normalization conditions given in \eqref{e:normal1}, with \(G = F^{N+1}\). Hence, plugging the form of \(\tilde{H}\) from \eqref{e:normal} into \eqref{e:ahlfors2}, we easily obtain
\[
\mathcal{A}(\tilde{H})_{\alpha\bba} = \frac{\partial a_{\beta}(z)}{\partial z_{\alpha}}.
\]
The proof is complete.
\end{proof}
Using Corollary \ref{cor:equality}, we can state a version of Huang--Lu--Tang--Xiao theorem \cite{huang2020boundary} for the CR Ahlfors tensor, which will be crucial for the next section.
\begin{corollary}\label{cor33}
Let \(U\) be a connected open neighborhood of a point \(p \in \heisn_{\ell}\) and \(H \colon U \to \C^N\) such that \(H(U\cap \heisn_{\ell}) \subset \heisN_{\ell'}\). Then the following are equivalent
\begin{enumerate}
\item \(H\) is CR transversal at \(p\) and \(\mathcal{A}(H) = 0\) in a neighborhood of~\(p\).
\item \(H\) is CR transversal at \(p\) and the Hermitian part of the CR Ahlfors derivative vanishes, \(\mathcal{A}(H)_{\alpha\bba} = 0\), in a neighborhood of~\(p\).
\item There exists \(\tau \in \mathrm{Aut}(\heisN)\) and \(\gamma \in \mathrm{Aut}(\heisn)\) such that 
	\begin{equation}\label{e:iso}
		\tau \circ H \circ \gamma^{-1} = (z,\phi,\psi,w)
	\end{equation}
where \(\phi\) and \(\psi\) are holomorphic maps near 0 with \(\ell'- \ell\) and \(N-n-\ell' + \ell\) components, respectively, satisfying \(\|\phi\| = \|\psi\|\).
\end{enumerate}
\end{corollary}
It is immediate that (3) implies (1). In fact, for the map of the form \eqref{e:iso}, the quotient \(Q\) is identically 1 and hence its CR Ahlfors derivative vanishes identically in a neighborhood of \(p\). That (1) implies (2) is trivial, and (2) implies (3) follows from the main result of Huang--Lu--Tang--Xiao \cite{huang2020boundary} together with \Cref{cor:equality}.

\section{Winkelmann hypersurfaces}
In this section, we discuss several basic properties of a Winkelmann hypersurface which are directly related
to the CR mapping problem. In particular, we deduce that the CR Ahlfors derivative is an invariant of the equivalent classes of CR maps from a hyperquadric into a Winkelmann hypersurfaces as well as of CR maps from a Winkelmann hypersurface into a hyperquadric of higher dimension,

For the Winkelmann hypersurface of signature $\ell'$ in \(\mathbb{C}^{n'+2}\), it's stability group at the origin is spanned by the following automorphisms. They are obtained by integrating the vector fields stated in Kruglikov \cite{kruglikov}, which vanish at $0$,
\begin{align*}
H(z,\zeta,w) & = (\lambda^2 u_1 z_1, \ldots, \lambda^2 u_{n'-1} z_{n'-1}, \lambda u_{n'} z_{n'}, \lambda^3 u_{n'} \zeta, \lambda^4 w), \\
S(z,\zeta, w) & = \left(z' + a z_{n'}, z_{n'}, \zeta + 2 i \langle
\bar a, z'\rangle_{\ell} + \left(r + i \| a \|^2_{\ell'}\right) z_{n'}, w\right),\\
R(z,\zeta, w) & = (U z', z_{n'}, \zeta, w),
\end{align*}
where $z=(z',z_{n'})=(z_1,\ldots, z_{{n'}-1},z_{n'}), \lambda > 0, |u_k|=1,  r \in \R, a\in \C^{{n'}-1}$, and $\sigma I_{{n'}-1,\ell'} = U I_{{n'}-1,\ell'} U^{\ast},\ \sigma \in \{-1,1\}$. The following proposition states properties of the CR automorphisms of the Winkelmann hypersurface.

\begin{proposition}\label{prop:41}
Let $\mathcal{W}_{\ell'}^{2{n'}+3}$ be a Winkelmann hypersurface and $\psi$ is a CR automorphism. Then the following holds:
\begin{enumerate}
\item \(\psi\) is a homothety of the pseudo-Hermitian structure \(\theta_{\mathcal{W}_{\ell'}^{2{n'}+3}} : = i \bar{\partial} \rho_{\mathcal{W}_{\ell'}^{2{n'}+3}}\), that is
\[ \psi^{\ast} \theta_{\mathcal{W}_{\ell}^{2{n'}+3}} = C \theta_{\mathcal{W}_{\ell'}^{2{n'}+3}},  \]
for some constant \(C\);
\item $\psi$ extends to an isometry of $\Omega$ with respect to the K\"ahler metric given by \eqref{e:winkelman-kahler}. 
\end{enumerate}
\end{proposition}
\begin{proof}
We can check these properties directly for each of the maps \(H,S,R\) above. Details are left to the readers.
\end{proof}
Proposition~\ref{prop:41} implies that the CR automorphism group of the Winkelmann hypersurface is equal to the CR M\"obius group with respect to the usual pseudo-Hermitian structure: The CR automorphisms are homotheties. Here, a CR diffeomorphism of pseudo-Hermitian manifold \((M,\theta)\) is CR M\"obius if it preserves the traceless part of the pseudo-Hermitian Ricci tensor and the pseudo-Hermitian torsion tensor. Equivalently, it is CR M\"obius if its CR Schwarzian derivative vanishes; see \cite{son2018} for the details.

Part (i) in Proposition~\ref{prop:41} exhibits a special property of the ``standard'' pseudo-Hermitian structure \(\theta_{\mathcal{W}_{\ell'}^{2{n'}+3}}\), namely, the homothety group coincides with the CR automorphism group. We point out, however, that this coincidence is in general stronger than what are needed for the invariant property of the CR Ahlfors and Schwarzian tensor with respect to compositions with CR automorphisms. Precisely, for such an invariance property to hold, we need that the CR M\"obius group coincides with the CR automorphism group and the last condition is satisfied also for other choices of pseudo-Hermitian structures, the ones which are CR M\"obius with respect to \(\theta_{\mathcal{W}_{\ell'}^{2{n'}+3}}\). All such structures can be explicitly determined as in the proposition below. Indeed, one can easily check that \(\theta_{\mathcal{W}_{\ell'}^{2{n'}+3}}\) is pseudo-Einstein (the Webster Ricci tensor is a scalar multiple of the Levi metric) and has vanishing pseudo-Hermitian torsion. Thus, a pseudo-Hermitian structure \(\theta\) on \(\mathcal{W}_{\ell'}^{2{n'}+3}\) is CR M\"obius with respect to \(\theta_{\mathcal{W}_{\ell'}^{2{n'}+3}}\) if and only if \(\theta\) is pseudo-Einstein with vanishing pseudo-Hermitian torsion.

\begin{proposition} Let \(\theta = e^{u} \theta_{\mathcal{W}_{\ell'}^{2{n'}+3}}\) be a pseudo-Hermitian structure on \(\mathcal{W}_{\ell'}^{2{n'}+3}\).  Then \(\theta\) is pseudo-Einstein with vanishing pseudo-Hermitian torsion if and only if
\begin{equation}\label{eq:u}
u = \log \left|c_0 + \sum_{k=1}^{n'} c_k z_k + c_{{n'}+2} w \right|,
\end{equation}
for some constants \(c_0, c_1,\dots, c_{{n'}+2}\).
\end{proposition}
\begin{proof} It follows from \cite{son2018} that \(\theta\) is pseudo-Einstein with vanishing pseudo-Hermitian torsion if and only if the following equation satisfies
\[
B(u) = 0
\]
on \(\mathcal{W}_{\ell}^{2n+3}\). To solve this equation, we first consider the CR structure on \(\C^{n+1} \times \R\) given by the following vector fields
\[
Z_{\aba} = \frac{\partial}{\partial \bar{z}_{\alpha}} - i \frac{\partial \varphi}{\partial \bar{z}_{\alpha}} \frac{\partial}{\partial t}, \quad \alpha= 1,\dots, {n'}+1,
\]
where
\[
\varphi(z,\zeta) = \frac{z_{n'} \bar{\zeta} - \bar{z}_{n'} \zeta}{2i} + |z_{n'}|^4 + \sum_{k=1}^{{n'}-1} \epsilon_k |z_k|^2.
\]
This is a Levi-nondegenerate CR structure. Let 
\[
\theta = \frac{1}{2} \left(dt - i \varphi_{\alpha} dz_{\alpha} + i \varphi_{\aba}d \bar{z}_{\alpha}\right)
\]
be a pseudo-Hermitian structure. Then the Reeb field is
\[
T = 2 \frac{\partial}{\partial t},
\]
and the Levi-form is
\[
d\theta = i\varphi_{\alpha\bba} dz_{\alpha} \wedge d\zba_{\beta}.
\]
Thus,
\[
h_{\alpha\bba} = -id\theta (Z_{\alpha}, Z_{\bba}) = \varphi_{\alpha\bba}.
\]
We have
\[
[Z_{\bba}, Z_{\alpha}] = 2i \varphi_{\alpha\bba} \frac{\partial }{\partial t} = i \langle Z_{\alpha}, Z_{\bba} \rangle T.
\]
Therefore, the Christoffel symbols of mixed type in the chosen CR frame are
\[
\Gamma_{\aba\beta}^{\gamma} = 0.
\]
The Christoffel symbols of pure type are 
\[
\Gamma_{\beta\alpha}^{\gamma} = h^{\gamma\bar{\sigma}} Z_{\beta} h_{\alpha\bar{\sigma}}
=
\begin{cases}
-8i \bar{z}_n, \quad & (\alpha,\beta,\gamma) = ({n'},{n'},{n'}+1), \\
0 & \text{otherwise}.
\end{cases}
\]

The M\"{o}bius equation is equivalent to the condition that \(u\) is CR pluriharmonic and 
\[
u_{\alpha,\beta} = 2 u_{\alpha} u_{\beta}.
\]
Locally, we can find a real-valued function \(v\) such that \(u+iv\) is CR. We put \(G = e^{-u-iv}\). Then \(G\) is also CR. By a direct calculation, we have
\[
G_{\alpha,\beta} = -G B(u)_{\alpha\beta} = 0.
\]
On the other hand, since \(G\) is CR and \(\theta\) has vanishing pseudo-Hermitian torsion, we have
\[
G_{0,\bar{\alpha}} = G_{\bar{\alpha},0} = 0,
\]
hence
\[
G_{0,\alpha} = -\frac{i}{{n'}+1} \left(G_{\beta,}{}^{\beta}{}_{\alpha} - G_{\bar{\beta},}{}^{\bar{\beta}}{}_{\alpha}\right) = -\frac{1}{{n'}+1} G_{\alpha,0}.
\]
Together with the commutation relation
\[
G_{0,\alpha} = G_{\alpha,0} + A_{\alpha}{}^{\bar{\beta}} G_{\bar{\beta}},
\]
we assert that \(G_{0,\alpha} = 0.\) Hence, \(G_0\) must be a constant.
Consider the function
\[
K = G - \frac{1}{2}G_0 w,
\]
where \(w = t + i\varphi\). As \(G_0\) is a constant, by the chain rule, we can easily see that 
\[
\frac{\partial K}{\partial t} = 0.
\]
That is, \(K\) is a holomorphic function of \(z_1,\dots, z_n, \zeta\). 

We can check directly that 
\[
w_{,\alpha\beta} = Z_{\beta}Z_{\alpha} w - \Gamma_{\beta\alpha}^{\gamma} Z_{\gamma} w = 0.
\]
Therefore,
\[
K_{,\alpha\beta} = G_{,\alpha\beta} - \frac{1}{2} G_0 w_{,\alpha\beta} = 0.
\]
In particular, for \((\alpha,\beta) = (n,n)\), we have
\[
0 = K_{,{n'}{n'}} = \frac{\partial^2 K}{\partial z_{n'}^2} - \Gamma^{{n'}+1}_{{n'}{n'}} \frac{\partial K}{\partial \zeta} = \frac{\partial^2 K}{\partial z_{n'}^2} + 8i \bar{z}_{n'} \frac{\partial K}{\partial \zeta}
\]
Applying \(\partial/\partial \bar{z}_{n'}\), we deduce that 
\[
\frac{\partial K}{\partial \zeta} = 0,
\]
and hence
\[
K_{,{n'}{n'}} = \frac{\partial^2 K}{\partial z_{n'}^2} = 0.
\]
Thus, \(K\) is a holomorphic function of \(z_1,\dots, z_{n'}\) and satisfies
\[
\frac{\partial^2 K}{\partial z_{\alpha}z_{\beta}} = K_{,\alpha\beta} = 0.
\]
In other words, \(K\) is a linear function of \(z_1,\dots, z_{n'}\):
\[
K = \sum_{k=1}^{n'} c_k z_k.
\] 
Thus
\[
G = c_0 + c_{{n'}+2} w + \sum_{k=1}^{n'} c_k z_k, \quad c_{{n'}+2} = G_0/2,
\]
and
\[
u = \log \left|c_0 + \sum_{k=1}^{n'} c_k z_k+ c_{{n'}+2} (t + i \varphi(z,\zeta))\right|.
\]
The pseudo-Einstein structure with vanishing torsion is of the form
\[
\tilde{\theta} = e^u \theta,
\]
with \(u\) being of the form above.

Finally, on \(\mathcal{W}^{2{n'}+3}_{\ell'}\), the solution to the M\"{o}bius equation is given by
\[
u = \log \left|c_0 + \sum_{k=1}^{n'} c_k z_k + c_{{n'}+2} w \right|.
\]
This completes the proof.
\end{proof}

\begin{corollary} Let \(\hat{\theta} = e^{u} \theta_{\mathcal{W}_{\ell'}^{2{n'}+3}}\) be a pseudo-Hermitian structure on \(\mathcal{W}_{\ell'}^{2{n'}+3}\), where $u$ is given as in \eqref{eq:u}. Let \(\phi\) be a CR automorphism of \(\mathcal{W}_{\ell'}^{2{n'}+3}\). Then \(\phi\) is M\"obius with respect to \(\hat{\theta}\). 
\end{corollary}
\begin{remark}\rm On a sphere or a hyperquadric, all pseudo-Hermitian structures which are M\"obius with respect to the standard structure \(\theta_{\mathbb{H}^{2n+1}_{\ell}}\) can be obtained by pulling back \(\theta_{\mathbb{H}^{2n+1}_{\ell}}\) via a suitable CR automorphism. In other words, \(\Aut(\mathbb{H}^{2n+1}_{\ell})\) acts on the set of pseudo-Hermitian structures via pulling back and the orbit of \(\theta_{\mathbb{H}^{2n+1}_{\ell}}\) coincides with the set of pseudo-Hermitian structures which are M\"obius with respect to \(\theta_{\mathbb{H}^{2n+1}_{\ell}}\).

On a Winkelmann hypersurface, the situation is a bit different: The orbit of the standard pseudo-Hermitian structure \(\theta_{\mathcal{W}_{\ell'}^{2{n'}+3}}\) under the action of \(\Aut(\mathcal{W}_{\ell'}^{2{n'}+3})\) is a \textit{proper} subset of the set of pseudo-Hermitian structures which are M\"obius with respect to \(\theta_{\mathcal{W}_{\ell'}^{2{n'}+3}}\).
\end{remark}
\begin{definition} Let \(H\) and \(\tilde{H}\) be two CR maps from \(\mathbb{H}_{\ell}^{2n+1}\) into
$\mathcal{W}^{2n'+3}_{\ell'}$. We say that \(H\) and \(\tilde{H}\) are equivalent if there exist CR automorphisms \(\phi \in \Aut(\mathcal{W}^{2n'+3}_{\ell'})\) and \(\gamma \in \Aut(\mathbb{H}_{\ell}^{2n+1})\) such that

\begin{equation}\label{equiv}
\tilde{H} = \phi \circ H \circ \gamma^{-1}.
\end{equation}

\end{definition}
In a similar way, we also define the equivalence for CR maps from a Winkelmann hypersurface into a hyperquadric. Details are left to the readers.

We conclude this section by the following invariant property of the CR Ahlfors tensor for CR maps from a hyperquadric into a Winkelmann hypersurface.
\begin{proposition}
If \(H\) and \(\tilde{H}\) are equivalent CR transversal maps from a hyperquadric into a Winkelmann hypersurface, then
\[ 
\mathcal{A}(H) = e^{u}\mathcal{A}(\tilde{H})
\]
where, the CR Ahlfors derivative are taken with respect to standard pseudo-Hermitian structures on \(\mathbb{H}_{\ell}^{2n+1}\) and \(\mathcal{W}^{2n'+3}_{\ell'}\), and \(u\) is determined by \(\gamma^{\ast} \theta = e^{u}\theta\). 
\end{proposition}
A similar statement holds for transversal CR immersions from a Winkelmann hypersurface into hyperquadric; details are left to the readers.
\section{CR maps from a hyperquadric into a Winkelmann hypersurface}

Based on \Cref{cor33}, we characterize extendability to a local isometry in terms of the vanishing of the Hermitian part of the CR Ahlfors tensor. This is possible due to the fact that the Winkelmann hypersurface can be embedded into a hyperquadric by a CR map with vanishing CR Ahlfors derivative.

\begin{theorem} Let \(U \subset \mathbb{H}^{2n+1}_{\ell}\) be an open subset and \(H\colon U \to\mathcal{W}_{\ell'}^{2n'+3}\). Let \(\mathcal{A}_{\alpha\bba}(H)\) be the Hermitian part of the CR Ahlfors tensor of \(H\) with respect to ``standard'' pseudo-Hermitian structures of the source and target. Then \(\mathcal{A}_{\alpha\bba}(H) = 0\) on \(U\) if and only if \(H\) extends to a local isometric embedding of the indefinite K\"ahler metrics.
\end{theorem}
\begin{proof} We first observe that the Winkelmann hypersurface can be embedded into a hyperquadric by a CR embedding with vanishing CR Ahlfors tensor. Specifically, the quadratic holomorphic embedding $\Phi \colon \mathbb{C}^{n'+2} \to \mathbb{C}^{n'+3}$ defined by
\begin{align*}
\Phi(z_1,\ldots, z_{n'},\zeta,w) = \left(z_1, \ldots, z_{n'-1},z_n^2,\frac{z_{n'} + i \zeta}{2}, \frac{z_{n'} - i \zeta}{2}, w \right),\\
\end{align*}
sends $\mathcal{W}^{2n'+3}_{\ell'}$ into $\mathbb{H}_{\ell'+1}^{2n'+5} \subset \mathbb{C}^{n'+3}$, transversally. Moreover, since 
\[ 
\rho_{\mathbb{H}_{\ell'+1}^{2n'+5}} \circ \Phi = \rho_{\mathcal{W}^{2n'+3}_{\ell'}},
\]
we immediately see that \(\mathcal{A}(\Phi) = 0\) .

Let $\widetilde{H} = \Phi \circ H$. Then $\widetilde{H}$ is a smooth CR map sending $\mathbb{H}_{\ell}^{2n+1}$ into $\mathbb{H}^{2n'+5}_{\ell'+1}$. By the chain rule \cite{lamel2021cr},
\[
\mathcal{A}(\widetilde{H}) = \mathcal{A}(H) + H^{\ast} \mathcal{A}(\Phi) = 0.
\]
Hence, by Corollary \ref{cor33}, $\widetilde{H}$ extends to a local holomorphic isometric embedding of \(\mathcal{U}^{n+1}_{\ell}\) into \(\mathcal{U}^{N+3}_{\ell'+1}\). Together with the explicit form of $\Phi$, we can easily deduce that $H$ extends to a local holomorphic isometric embedding. This completes the proof.
\end{proof}

\section{Examples}
In this section, we give various examples of CR maps from a hyperquadric into a Winkelmann hypersurface. We point out that, similar to the hyperquadric case, there are examples of CR maps from a hyperquadric into a Winkelmann hypersurface whose geometric ranks vanish, but are not equivalent.
\begin{example}\rm We first observe that the map
\[
(z_1,\dots, z_{n},w) \mapsto (z_1,\dots, z_{n-1},\sqrt{z_n}, 0 ,w),
\]
has a singularity at the origin. Outside its singular locus, it sends \(\mathbb{H}^{2n+1}_{\ell}\) into \(\mathcal{W}^{2n+3}_{\ell+3}\). Since both the source and target are homogeneous, we can use their automorphism groups to obtain the map 
\begin{equation}\label{irrmap}
I(z_1,z_2, \dots, w) = \left(z_1,\dots, z_{n-1}, \sqrt{1+z_n} -1, 4i \left(\sqrt{1+z_n} -1 - z_n\right) w\right),
\end{equation}
which is holomorphic at the origin and sends \(\mathbb{H}^{2n+1}_{\ell}\) into \(\mathcal{W}^{2n+3}_{\ell+1}\). Details are left to the reader. This map is irrational in the standard coordinates. Moreover, it extends to a local holomorphic isometric embedding from \(\mathcal{U}^{n+1}_{\ell}\) into \(\Omega^{n+2}_{\ell+1}\).
\end{example}

\begin{example}\rm 
For \(\epsilon \in \{-1,1\}\), the quadratic map \(R\colon \C^{n+1} \to \C^{n+2}\)	defined by
\begin{align}\label{e:rat}
R & (z_1, z_2, \dots, z_n, w) \notag \\ & = \left((1 + \epsilon z_n)z_1, \dots, (1+ \epsilon z_n) z_{n-1}, z_n, w(\epsilon + z_n) - i z_n(1+2 \epsilon z_n), w(1 + \epsilon z_n)\right)
\end{align}
sends $\mathbb{H}^{2n+1}_{\ell}$ into $\mathcal{W}^{2n+3}_{\ell +1}$. Along \(\mathbb{H}^{2n+1}_{\ell}\), it is CR transversal precisely when \(z_n \ne - \epsilon\). Furthermore, it has full rank precisely when \(z_n \ne -\epsilon\). Moreover, since
\[
\rho_{\mathcal{W}_{\ell+1}^{2n+3}} \circ R = | \epsilon + z_n|^2 \rho_{\mathbb{H}_{\ell}^{2n+1}},
\]
we see that \(R\) is a local isometric embedding of \(\mathcal{U}^{n+1} \setminus \{ z_n = - \epsilon\}\) into $\Omega_+$. Clearly, 
\[
\mathcal{A}(R)_{\alpha\bba} = \langle i \partial \bar{\partial} \log | \epsilon + z_n|^2 , Z_{\alpha} \wedge Z_{\bba} \rangle = 0, 
\]
i.e., \(R\) has vanishing geometric rank.
\end{example}
That \(I\) and \(R\) are not equivalent is immediate, since \(I\) is irrational while \(R\) and all CR automorphisms of the source and target are rational in the standard coordinates.
\begin{example}\rm When \(n=1\), the map in the previous example specializes to
\begin{equation}\label{e:re}
R_{\epsilon}(z,w) = (z, w(\epsilon + z) - i z - 2i \epsilon z^2, w(1+ \epsilon z)).
\end{equation}
It is interesting to note that \eqref{e:re} defines a map from \(\heis\) into \(\mathcal{W}^5_1\) for arbitrary real value of \(\epsilon\). Thus, we obtain a 1-parameter family of quadratic maps. Moreover, it can be checked that the Hermitian part of the CR Ahlfors tensor is 
\[
\mathcal{A}(R_{\epsilon})_{1\bar{1}} = \frac{1 - \epsilon^2}{1 + |z|^2 + 2\epsilon \Re(z)},
\]
which vanishes identically if and only if \(\epsilon \in \{-1,1\}\). Hence, for \(\epsilon \not\in \{-1,1\}\), the map \(R_{\epsilon}\) does not extend to an isometry, even locally.
\end{example}

It is worth mentioning that \(R_{\epsilon}\) is a sub-family of a 2-parameter family of rational maps
\[
R_{\epsilon,\mu} (z,w) = \left(\frac{z}{1 - \mu w}, \frac{-2i z^2 + \epsilon w(1-\mu w) + z(-i + w +3i\mu w)}{(1-\mu w)^2}, \frac{w(1+ \epsilon z - 2 \mu w)}{(1- \mu w)^2}\right),
\]
which, for any real numbers $\epsilon$ and $\mu$, sends \(\heis\) into \(\mathcal{W}^5_1\).
\begin{example}\rm 
The map
\[
(z,w) \mapsto (0,\phi(z,w),0),
\]
where \(\phi(z,w)\) is an arbitrary CR map in a neighborhood of the origin, sends a neighborhood of the origin in \(\heis\) into \(\mathcal{W}^5_1\). This map is nowhere CR transversal.
\end{example}
We conjecture that an arbitrary CR map from \(\heis\) into \(\mathcal{W}^5_1\) must be equivalent to one of the maps appearing in this section.


\begin{thebibliography}{10}

\bibitem{Alexander77}
H.~Alexander.
\newblock {Proper holomorphic mappings in {$\mathbb C^{n}$}}.
\newblock {\em Indiana Univ. Math. J.}, 26(1):137--146, 1977.

\bibitem{AHJY16}
Jared Andrews, Xiaojun Huang, Shanyu Ji, and Wanke Yin.
\newblock {Mapping in {$\mathbb B^n$} into {$\mathbb B^{3n-3}$}}.
\newblock {\em Comm. Anal. Geom.}, 24(2):279--300, 2016.

\bibitem{BEH08}
M.~Salah Baouendi, Peter Ebenfelt, and Xiaojun Huang.
\newblock {Super-rigidity for {CR} embeddings of real hypersurfaces into
hyperquadrics}.
\newblock {\em Adv. Math.}, 219(5):1427--1445, 2008.

\bibitem{BEH11}
M.~Salah Baouendi, Peter Ebenfelt, and Xiaojun Huang.
\newblock {Holomorphic mappings between hyperquadrics with small signature
difference}.
\newblock {\em Amer. J. Math.}, 133(6):1633--1661, 2011.

\bibitem{baouendi2005super}
M~Salah Baouendi and Xiaojun Huang.
\newblock Super-rigidity for holomorphic mappings between hyperquadrics with
positive signature.
\newblock {\em Journal of Differential Geometry}, 69(2):379--398, 2005.

\bibitem{cartan1935domaines}
Elie Cartan.
\newblock Sur les domaines born\'{e}s homog\`enes de l'espace den variables
complexes.
\newblock {\em Abh. Math. Sem. Univ. Hamburg}, 11(1):116--162, 1935.

\bibitem{CJX06}
Zhihua Chen, Shanyu Ji, and Dekang Xu.
\newblock {Rational holomorphic maps from {$\mathbb B^2$} into {$\mathbb B^N$}
with degree 2}.
\newblock {\em Sci. China Ser. A}, 49(11):1504--1522, 2006.

\bibitem{CJY19}
Xiaoliang Cheng, Shanyu Ji, and Wanke Yin.
\newblock Mappings between balls with geometric and degeneracy rank two.
\newblock {\em Sci. China Math.}, 62(10):1947--1960, 2019.

\bibitem{chern1974}
SS~Chern and JK~Moser.
\newblock Real hypersurfaces in complex manifolds.
\newblock {\em Acta Mathematica}, 133(1):219--271, 1974.

\bibitem{dangelo1988proper}
John~P. D'Angelo.
\newblock Proper holomorphic maps between balls of different dimensions.
\newblock {\em Michigan Math. J.}, 35(1):83--90, 1988.

\bibitem{DAngelo07a}
John~P. D'Angelo.
\newblock {The {$X$}-varieties for {CR} mappings between hyperquadrics}.
\newblock {\em Asian J. Math.}, 11(1):89--102, 2007.

\bibitem{DAngelo15}
John~P. D'Angelo.
\newblock {C{R} complexity and hyperquadric maps}.
\newblock In {\em {Analysis and geometry}}, volume 127 of {\em {Springer Proc.
Math. Stat.}}, pages 17--34. Springer, Cham, 2015.

\bibitem{DAngeloBook21}
John~P. D'Angelo.
\newblock {\em Rational sphere maps}, volume 341 of {\em Progress in
Mathematics}.
\newblock Birkh\"{a}user/Springer, Cham, [2021] \copyright 2021.

\bibitem{doubrov2021}
Boris Doubrov, Alexandr Medvedev, and Dennis The.
\newblock Homogeneous levi non-degenerate hypersurfaces in {$\mathbb C^3$}.
\newblock {\em Mathematische Zeitschrift}, 297(1-2):669--709, 2021.

\bibitem{faran1982maps}
James~J. Faran.
\newblock Maps from the two-ball to the three-ball.
\newblock {\em Inventiones mathematicae}, 68(3):441--475, 1982.

\bibitem{GN21a}
Yun Gao and Sui-Chung Ng.
\newblock Local orthogonal maps and rigidity of holomorphic mappings between
real hyperquadrics.
\newblock 2021.

\bibitem{GLV14}
Dusty Grundmeier, Ji{\v r}{\'i} Lebl, and Liz Vivas.
\newblock {Bounding the rank of {H}ermitian forms and rigidity for {CR}
mappings of hyperquadrics}.
\newblock {\em Math. Ann.}, 358(3-4):1059--1089, 2014.

\bibitem{Hamada05}
Hidetaka Hamada.
\newblock {Rational proper holomorphic maps from {$\mathbb B^n$} into {$\mathbb
B^{2n}$}}.
\newblock {\em Math. Ann.}, 331(3):693--711, 2005.

\bibitem{Huang99}
Xiaojun Huang.
\newblock {On a linearity problem for proper holomorphic maps between balls in
complex spaces of different dimensions}.
\newblock {\em J. Differential Geom.}, 51(1):13--33, 1999.

\bibitem{Huang03}
Xiaojun Huang.
\newblock On a semi-rigidity property for holomorphic maps.
\newblock {\em Asian J. Math.}, 7(4):463--492, 2003.

\bibitem{huang2001mapping}
Xiaojun Huang and Shanyu Ji.
\newblock Mapping {$\mathbb B^n$} into {$\mathbb B^{2n-1}$}.
\newblock {\em Inventiones mathematicae}, 145(2):219--250, 2001.

\bibitem{HJX06}
Xiaojun Huang, Shanyu Ji, and Dekang Xu.
\newblock {A new gap phenomenon for proper holomorphic mappings from {$B^n$}
into {$B^N$}}.
\newblock {\em Math. Res. Lett.}, 13(4):515--529, 2006.

\bibitem{HJY14}
Xiaojun Huang, Shanyu Ji, and Wanke Yin.
\newblock {On the third gap for proper holomorphic maps between balls}.
\newblock {\em Math. Ann.}, 358(1-2):115--142, 2014.

\bibitem{huang2020boundary}
Xiaojun Huang, Jin Lu, Xiaomin Tang, and Ming Xiao.
\newblock Boundary characterization of holomorphic isometric embeddings between
indefinite hyperbolic spaces.
\newblock {\em Advances in Mathematics}, 374:107388, 2020.

\bibitem{Ji10}
Shanyu Ji.
\newblock {A new proof for {F}aran's theorem on maps between {$\mathbb B^2$}
and {$\mathbb B^3$}}.
\newblock In {\em {Recent advances in geometric analysis}}, volume~11 of {\em
{Adv. Lect. Math. (ALM)}}, pages 101--127. Int. Press, Somerville, MA, 2010.

\bibitem{kruglikov}
Boris Kruglikov.
\newblock Submaximally symmetric {CR}-structures.
\newblock {\em J. Geom. Anal.}, 26(4):3090--3097, 2016.

\bibitem{lamel2021cr}
Bernhard Lamel and Duong~Ngoc Son.
\newblock The {CR} ahlfors derivative and a new invariant for spherically
equivalent {CR} maps.
\newblock {\em Annales de l'Institut Fourier}, 71(5):2137--2167, 2021.

\bibitem{Lebl11}
Ji{\v r}{\'i} Lebl.
\newblock {Normal forms, {H}ermitian operators, and {CR} maps of spheres and
hyperquadrics}.
\newblock {\em Michigan Math. J.}, 60(3):603--628, 2011.

\bibitem{LP11}
Ji{\v r}{\'i} Lebl and Han Peters.
\newblock {Polynomials constant on a hyperplane and {CR} maps of
hyperquadrics}.
\newblock {\em Mosc. Math. J.}, 11(2):285--315, 407, 2011.

\bibitem{lee1988pseudo}
John~M Lee.
\newblock Pseudo-einstein structures on {CR} manifolds.
\newblock {\em American Journal of Mathematics}, 110(1):157--178, 1988.

\bibitem{Reiter16a}
Michael Reiter.
\newblock {Classification of holomorphic mappings of hyperquadrics from
{$\mathbb{C}^2$} to {$\mathbb{C}^3$}}.
\newblock {\em J. Geom. Anal.}, 26(2):1370--1414, 2016.

\bibitem{son2018}
Duong~Ngoc Son.
\newblock The {S}chwarzian derivative and {M{\"o}}bius equation on strictly
pseudo-convex {CR} manifolds.
\newblock {\em Communications in Analysis and Geometry}, 26(2):237--269, 2018.

\bibitem{stowe}
Stowe, Dennis. 
\newblock {An Ahlfors derivative for conformal immersions.} 
\newblock {\em The Journal of Geometric Analysis}, 25(1):592--615, 2015.

\bibitem{tanaka1962}
Noboru Tanaka.
\newblock On the pseudo-conformal geometry of hypersurfaces of the space
of n complex variables.
\newblock {\em J. Math. Soc. Japan,} 14, 397-429 (1962).

\bibitem{Webster79b}
S.~M. Webster.
\newblock {On mapping an {$n$}-ball into an {$(n+1)$}-ball in complex spaces}.
\newblock {\em Pacific J. Math.}, 81(1):267--272, 1979.

\bibitem{winkelmann}
J\"{o}rg Winkelmann.
\newblock {\em The classification of three-dimensional homogeneous complex
manifolds}, volume 1602 of {\em Lecture Notes in Mathematics}.
\newblock Springer-Verlag, Berlin, 1995.
\end{thebibliography}
\end{document}